\documentclass[11pt]{amsart}
\usepackage{latexsym}
\usepackage{amsmath,amssymb,amsfonts}
\usepackage{graphics}
\usepackage[all]{xy}

\def\Bbb{\mathbb}

\def\eea{\end{eqnarray*}}

\newtheorem{defn}{Definition}
\newtheorem{thm}{Theorem}[section]

\newtheorem{lem}[thm]{Lemma}

\newenvironment{rmk}{\mbox{ }\\{\bf  Remark}\mbox{ }}{
\hfill $\Box$\mbox{}\bigskip}

\begin{document}
\renewcommand{\theequation}{\thesection.\arabic{equation}}

\title[$G$-monopole invariants on 4-manifolds]{$G$-monopole invariants on some connected sums of 4-manifolds}

\author{Chanyoung Sung}
\date{\today}

\address{Dept. of mathematics education \\
Korea national university of education}
\email{cysung@kias.re.kr}

\begin{abstract}
On a smooth closed oriented $4$-manifold $M$ with a smooth action
of a finite group $G$ on a  Spin$^c$ structure, $G$-monopole
invariant is defined by ``counting" $G$-invariant solutions of
Seiberg-Witten equations for any $G$-invariant Riemannian metric
on $M$.

We compute $G$-monopole invariants on some $G$-manifolds. For example, the connected sum of $k$ copies of a 4-manifold with nontrivial mod 2 Seiberg-Witten invariant has nonzero $\Bbb Z_k$-monopole invariant mod 2, where the $\Bbb Z_k$-action is given by cyclic permutations of $k$ summands.
\end{abstract}
\maketitle
\setcounter{section}{0}
\setcounter{equation}{0}

\section{Introduction}
Let $M$ be a smooth closed oriented manifold of dimension 4.
A second cohomology class of $M$ is called a \emph{monopole class}
if it arises as the first Chern class of a Spin$^c$ structure for
which the Seiberg-Witten equations
$$\left\{
\begin{array}{ll} D_A\Phi=0\\
  F_{A}^+=\Phi\otimes\Phi^*-\frac{|\Phi|^2}{2}\textrm{Id},
\end{array}\right.
$$
admit a solution for every choice of a Riemannian metric. Clearly
a basic class, i.e. the first Chern class of a Spin$^c$ structure with a nonzero Seiberg-Witten invariant is a monopole class. However, ordinary Seiberg-Witten invariants which are gotten by the intersection
theory on the moduli space of solutions $(A,\Phi)$ of the above equations is trivial in many
important cases, for example connected sums of 4-manifolds with $b_2^+>0$.

Bauer and Furuta \cite{BF, bau} made a breakthrough in detecting a
monopole class on certain connected sums of 4-manifolds. Their new
brilliant idea is to generalize the Pontryagin-Thom construction
to a proper map between infinite-dimensional spaces, which is the
sum of a linear Fredholm map and a compact map, and take some sort of a
stably-framed bordism class of the Seiberg-Witten moduli space as an
invariant. However its applications are still limited in that
this new invariant which is expressed as a stable cohomotopy class
is difficult to compute, and we are seeking after further refined
invariants of the Seiberg-Witten moduli space.

In the meantime, sometimes we need a solution of the
Seiberg-Witten equations for a specific metric rather than any
Riemannian metric. The case we have in mind is the one when a
manifold $M$ and its Spin$^c$ structure $\frak{s}$ admit a smooth
orientation-preserving action by a compact Lie group $G$ and we
are concerned with finding a solution of the Seiberg-Witten
equations for any $G$-invariant metric.

Thus for a $G$-invariant metric on $M$ and a $G$-invariant
perturbation of the Seiberg-Witten equations, we consider the
\emph{$G$-monopole moduli space} $\frak{X}$ consisting of their
$G$-invariant solutions modulo gauge equivalence. One can easily
see that the ordinary moduli space $\frak{M}$ is acted on by $G$,
and $\frak{X}$ turns out to be a subset of its $G$-fixed point
set. The intersection theory on $\frak{X}$ will give the
\emph{$G$-monopole invariant} $SW^{G}_{M,\frak{s}}$ defined first
by Y. Ruan \cite{ruan}, which encodes the information of the given
$G$-action along with $M$, and may be sometimes sharper than the
ordinary Seiberg-Witten invariant $SW_{M,\frak{s}}$. To be
precise, we need the dimension $b_2^+(M)^G$ of the maximal
dimension of subspaces of $G$-invariant 2nd cohomology classes of $M$,
where the intersection form is positive-definite to be bigger than
1. In view of this, the following definition is relevant :
\begin{defn}
Suppose that $M$ admits a smooth action by a compact Lie group $G$
preserving the orientation of $M$.

A second cohomology class of $M$ is called a $G$-\emph{monopole
class} if it arises as the first Chern class of a $G$-equivariant
Spin$^c$ structure for which the Seiberg-Witten equations admit a
$G$-invariant solution for every $G$-invariant Riemannian metric of $M$.
\end{defn}
When the $G$-monopole invariant is nonzero, its first Chern class
has to be a $G$-monopole class. As explain in \cite{sung3}, the
cases we are aiming at are those for finite $G$.  If a compact
connected Lie group $G$  has positive dimension and is not a torus
$T^n$, then $G$ contains a Lie subgroup isomorphic to $S^3$ or
$S^3/\Bbb Z_2$, and hence $M$ admitting an effective action of
such $G$ must have a $G$-invariant metric of positive scalar
curvature by the well-known Lawson-Yau theorem \cite{law-yau}.
Therefore when $b_2^+(M)^G>1$, $M$ has no $G$-monopole class for
such $G$. On the other hand, the Seiberg-Witten invariants of a
4-manifold with an effective $S^1$ action were extensively studied
by S. Baldridge \cite{bal1, bal2, bal3}.

Using $G$-monopole invariants, we find $G$-monopole classes in some
connected sums which have vanishing Seiberg-Witten invariants :
\begin{thm}\label{firstth}
Let $M$ and $N$ be  smooth closed oriented connected 4-manifolds satisfying
$b_2^+(M)> 1$ and $b_2^+(N)=0$, and $\bar{M}_k$ for any $k\geq 2$ be the connected sum
$M\#\cdots \#M\# N$ where there are $k$ summands of $M$.

Suppose that a finite group $G$ with $|G|=k$ acts effectively on $N$ in a smooth
orientation-preserving way such that it is free or has at least one fixed point, and that $N$ admits a Riemannian metric of positive scalar curvature invariant under the $G$-action and a $G$-equivariant Spin$^c$ structure $\frak{s}_N$ with $c_1^2(\frak{s}_N)=-b_2(N)$.

Define a $G$-action on $\bar{M}_k$ induced from that of $N$
permuting $k$ summands of $M$ glued along
a free orbit in $N$, and let $\bar{\frak{s}}$ be the
Spin$^c$ structure on $\bar{M}_k$ obtained by gluing $\frak{s}_N$ and a Spin$^c$
structure $\frak{s}$ of $M$.

Then for any $G$-action on $\bar{\frak{s}}$ covering the above
$G$-action on $\bar{M}_k$, $SW^{G}_{\bar{M}_k,\bar{\frak{s}}}$ mod 2 is nontrivial if
$SW_{M,\frak{s}}$ mod 2 is nontrivial.
\end{thm}
The precise computation of  $SW^{G}_{\bar{M}_k,\bar{\frak{s}}}$ mod 2 will be given in Section 3, and we will also give some examples of such $N$ in the last section. The condition on $N$ may be generalized a bit more.

This article is a refined publish version of original results announced in the archive \cite{sung4}. In a subsequent paper \cite{sung5}, we will use $G$-monopole invariants to detect smooth exotic actions of finite groups on 4-manifolds. The existence of a $G$-monopole class also has applications to Riemannian geometry such as $G$-invariant Einstein metrics and $G$-Yamabe invariant, which are dealt with in \cite{sung3}.

\section{$G$-monopole invariant}
Let $M$ be  a  smooth closed oriented 4-manifold.  Suppose that a
compact Lie group $G$ acts on $M$ smoothly preserving the
orientation, and this action lifts to an action on a Spin$^c$
structure $\mathfrak{s}$ of $M$. Once there is a lifting, any
other lifting differs from it by an element of $Map(G\times M,
S^1)$. We fix a lifting and put a $G$-invariant Riemannian metric
$g$ on $M$. Then the associated spinor bundles $W_{\pm}$ are also
$G$-equivariant, and we let $\Gamma(W_{\pm})^G$ be the set of its
$G$-invariant sections. When we put $G$ as a superscript on a set,
we always mean the subset consisting of its $G$-invariant
elements. Thus $\mathcal{A}(W_+)^G$ is the space of $G$-invariant
connections on $\det (W_+)$, which is identified as the space of
$G$-invariant purely-imaginary valued 1-forms
$\Gamma(\Lambda^1(M;i\Bbb R))^G$, and $\mathcal{G}^G=
Map(M,S^1)^G$ is the group of $G$-invariant gauge transformations.

The perturbed Seiberg-Witten equations give a smooth map $$H:
\mathcal{A}(W_+)^G\times \Gamma(W_+)^G\times \Gamma(\Lambda^2_+(M;i\Bbb
R))^G\rightarrow \Gamma(W_-)^G\times \Gamma(\Lambda^2_+(M;i\Bbb R))^G$$ defined by
$$H(A,\Phi,\varepsilon)=(D_A\Phi,
F_A^+-\Phi\otimes\Phi^*+\frac{|\Phi|^2}{2}\textrm{Id}+\varepsilon),$$
where the domain and the range are endowed with $L_{l+1}^2$ and
$L_l^2$  Sobolev norms for a positive integer $l$ respectively,
and $D_A$ is a Spin$^c$ Dirac operator. The $G$-monopole moduli
space $\frak{X}_\varepsilon$ for a perturbation $\varepsilon$ is
then defined as
$$\frak{X}_\varepsilon:=H^{-1}_\varepsilon(0)/ \mathcal{G}^G$$ where $H_\varepsilon$ denotes
$H$ restricted to $\mathcal{A}(W_+)^G\times \Gamma(W_+)^G\times\{\varepsilon\}$.

In the followings, we give a detailed proof that $\frak{X}_\varepsilon$ for generic $\varepsilon$ and finite $G$ is a smooth compact manifold, because some statements in \cite{ruan, cho}
are incorrect or without proofs.
\begin{lem}\label{saveyou}
The quotient map $$p : (\mathcal{A}(W_+)^G\times (\Gamma(W_+)^G-\{0\}))/ \mathcal{G}^G\rightarrow (\mathcal{A}(W_+)^G\times (\Gamma(W_+)^G-\{0\}))/
\mathcal{G}$$ is bijective, and hence $\mathfrak{X}_\varepsilon$ is a subset of the
ordinary Seiberg-Witten moduli space $\frak{M}_\varepsilon$.
\end{lem}
\begin{proof}
Obviously $p$ is surjective, and to show that $p$ is injective, suppose that  $(A_1,\Phi_1)$ and
$(A_2,\Phi_2)$ in $\mathcal{A}(W_+)^G\times (\Gamma(W_+)^G-\{0\})$ are
equivalent under $\gamma\in\mathcal{G}$. Then $$A_1=A_2-2d\ln \gamma, \ \ \ \textrm{and}\ \ \
\Phi_1=\gamma\Phi_2.$$ By the first equality, $d\ln \gamma$ is $G$-invariant.

Let $S$ be the subset of $M$ where $\gamma$ is $G$-invariant. By the continuity of $\gamma$, $S$ must be a closed subset. Since $S$ contains a nonempty subset $$\{x\in M| \Phi_1(x)\ne 0\},$$ $S$ is nonempty.
It suffices to show that ${S}$ is open. Let $x_0\in {S}$. Then we have that for any $g\in G$, $$g^*\ln\gamma(x_0)=\ln\gamma(x_0), \ \ \ \textrm{and}\ \ \ g^*d\ln\gamma=d\ln\gamma,$$ which implies that $g^*\ln\gamma=\ln\gamma$ on an open neighborhood of $x_0$ on which $g^*\ln\gamma$ and $\ln\gamma$ are well-defined. By the compactness of $G$, there exists an open neighborhood of $x_0$ on which $g^*\ln\gamma$ is well-defined for all $g\in G$, and $\ln \gamma$ is $G$-invariant. This proves the openness of $S$.
\end{proof}

As in the ordinary Seiberg-Witten moduli space, the transversality is
obtained by a generic perturbation $\varepsilon$ :
\begin{lem}
$H$ is a submersion at each $(A,\Phi,\varepsilon)\in H^{-1}(0)$ for nonzero
$\Phi$.
\end{lem}
\begin{proof}
Obviously $d{H}$ restricted to the last factor of the domain is onto
the last factor of the range. Using the surjectivity in the ordinary
setting,  for any element $\psi\in \Gamma(W_-)^G$, there exists
an element  $(a,\varphi)\in  \mathcal{A}(W_+)\times \Gamma(W_+)$
such that $d{H}(a,\varphi,0)=\psi$. The average
$$(\tilde{a},\tilde{\varphi}):=\int_G h^*(a,\varphi)\
d\mu(h):=( \int_G h^*a\ d\mu(h) , \int_G h^*\varphi\ d\mu(h))$$
using a unit-volume $G$-invariant metric on $G$ is an element of
$\mathcal{A}(W_+)^G\times \Gamma(W_+)^G$. It follows from the
smoothness of the $G$-action that every $h^*(a,\varphi)$ and hence
$(\tilde{a},\tilde{\varphi})$ belong to the same Sobolev space as
$(a,\varphi)$.
Moreover it still satisfies
\begin{eqnarray*}
d{H}(\tilde{a},\tilde{\varphi},0)&=& \int_G dH (h^*(a,\varphi,0))\
d\mu(h)\\ &=& \int_G h^* dH ((a,\varphi,0))\
d\mu(h)\\ &=& \int_G h^* \psi\
d\mu(h)\\ &=& \psi,
\end{eqnarray*}
where we used the fact that $d{H}$ is a $G$-equivariant differential operator. This completes the proof.
\end{proof}

Assuming that $b_2^{+}(M)^G$ is nonzero,
$\mathfrak{X}_\varepsilon$ consists of irreducible solutions. By
the above lemma, $\cup_{\varepsilon}\mathfrak{X}_\varepsilon$ is a
smooth submanifold, and applying Smale-Sard theorem to the
projection map onto $\Gamma(\Lambda^2_+(M;i\Bbb R))^G$,
$\mathfrak{X}_\varepsilon$  for generic $\varepsilon$ is also
smooth. (Nevertheless $\frak{M}_\varepsilon$ for that $\varepsilon$ may not be smooth in general. Its obstruction is explained in \cite{cho}.)   From now on, we will always assume that a generic
$\varepsilon$ is chosen so that $\frak{X}_\varepsilon$ is smooth,
and often omit the notation of $\varepsilon$, if no confusion
arises.

Its dimension and orientability can be obtained in the same way as the ordinary Seiberg-Witten moduli space. The linearization $dH$ is deformed by a homotopy to
$$d^++2d^* : \Gamma(\Lambda^1)^G\rightarrow
\Gamma(\Lambda^0\oplus\Lambda^2_+)^G$$ and $$D_A :
\Gamma(W_+)^G\rightarrow \Gamma(W_-)^G$$ so that the virtual dimension of  $\mathfrak{X}$ is equal to $$\dim H_1(M;\Bbb R)^G-b_2^+(M)^G-1+2(\dim_{\Bbb C}(\ker D_A)^G-\dim_{\Bbb C}(\textrm{coker} D_A)^G),$$
and its orientation can be assigned by fixing the homology orientation of $H_1(M;\Bbb R)^G$ and $H_2^+(M;\Bbb R)^G$. When $G$ is finite, one can use Lefschetz-type formula to explicitly compute the last term $\textrm{ind}^G D_A$ in the above formula. For more details, one may consult \cite{cho}.

\begin{thm}\label{cpt}
When $G$ is finite, $\mathfrak{X}_\varepsilon$ for any $\varepsilon$ is compact.
\end{thm}
\begin{proof}
Following the proof for the ordinary Seiberg-Witten moduli space, we need the $G$-equivariant version of the gauge fixing lemma.
\begin{lem}
Let $\frak{L}$ be a $G$-equivariant complex line bundle over $M$  with a hermitian metric, and $A_0$ be a fixed $G$-invariant smooth unitary connection on it.

Then for any $l\geq 0$ there are constants $K, C>0$ depending on $A_0$ and $l$ such that for any $G$-invariant $L^2_l$ unitary connection $A$ on $\frak{L}$ there is a $G$-invariant $L^2_{l+1}$ change of gauge $\sigma$ so that $\sigma^*(A)=A_0+\alpha$ where $\alpha\in L^2_l(T^*M\otimes i\Bbb R)^G$ satisfies
$$d^*\alpha=0,\ \ \ \textrm{and}\ \ \ \ ||\alpha||^2_{L^2_l}\leq C||F^+_A||^2_{L^2_{l-1}}+K.$$
\end{lem}
\begin{proof}
We know that a gauge-fixing $\sigma$ with the above estimate always exists, but we need to prove the existence of $G$-invariant $\sigma$.
Write $A$ as $A_0+a$ where $a\in L^2_l(T^*M\otimes i\Bbb R)^G$. Let $a=a^{harm}+df+d^*\beta$ be the Hodge decomposition. By the $G$-invariance of $a$, so are  $a^{harm}, df$, and $d^*\beta$.
Applying the ordinary gauge fixing lemma to $A_0+d^*\beta$, we have $$||d^*\beta||^2_{L^2_l}\leq C'||F^+_{A_0+d^*\beta}||^2_{L^2_{l-1}}+K'=C'||F^+_A||^2_{L^2_{l-1}}+K'$$ for some constants $C',K'>0$.
Defining a $G$-invariant $i\Bbb R$-valued function $f_{av}=\frac{1}{|G|}\sum_{g\in G}g^*f$, we have
$$df=\frac{1}{|G|}\sum_{g\in G}g^*df=d(f_{av})=-d\ln \exp(-{f_{av})},$$ and hence $df$ can be gauged away by a $G$-invariant gauge transformation $\exp(-f_{av})$.
Write $a^{harm}$ as $(n|G|+m)a^{h}$ for $m\in [0,|G|)$ and an integer $n\geq 0$, where $a^h\in H^1(M;\Bbb Z)^G$ is not a positive multiple of another element of $H^1(M;\Bbb Z)^G$. There exists $\frak{g}\in \mathcal{G}$ such that $a^h=-d\ln \frak{g}$. In general $\frak{g}$ is not $G$-invariant, but $$|G|a^h=\sum_{g\in G}g^*a^h=-d\ln \prod_{g\in G}g^*\frak{g},$$ and hence $n|G|a^h$ can be gauged away by a $G$-invariant gauge transformation $\prod_{g\in G}g^*\frak{g}^n$.
In summary, $A_0+a$ is equivalent to $A_0+ma^{h}+d^*\beta$ after a $G$-invariant gauge transformation, and
\begin{eqnarray*}
||ma^{h}+d^*\beta||^2_{L^2_l}&\leq& (||ma^{h}||_{L^2_l}+||d^*\beta||_{L^2_l})^2\\ &\leq& |G|^2||a^{h}||_{L^2_l}^2+2|G|||a^{h}||_{L^2_l}||d^*\beta||_{L^2_l}+||d^*\beta||_{L^2_l}^2\\ &\leq& 3|G|^2||a^{h}||^2_{L^2_l}+ 3||d^*\beta||_{L^2_l}^2  \\   &=& K''+3C'||F^+_A||^2_{L^2_{l-1}}+3K'
\end{eqnarray*}
for a constant $K''>0$. This completes the proof.
\end{proof}
Now the rest of the compactness proof proceeds in the same way as the ordinary case using the Weitzenb\"ock formula and standard elliptic and Sobolev estimates. For details the readers are referred to \cite{morgan}.
\end{proof}

\begin{rmk}
If $G$ is not finite, $\mathfrak{X}_\varepsilon$ may not be compact.

For example, consider $M=S^1\times Y$ with the trivial Spin$^c$
structure and its obvious $S^1$ action, where $Y$ is a closed
oriented 3-manifold. For any $n\in\Bbb Z$, $n d\theta$ where
$\theta$ is the coordinate on $S^1$ is an $S^1$-invariant
reducible solution. Although $n d\theta$ is gauge equivalent to 0,
but never via an $S^1$-invariant gauge transformation which is an
element of the pull-back of $C^\infty(Y,S^1)$. Therefore as
$n\rightarrow \infty$, $n d\theta$ diverges modulo
$\mathcal{G}^{S^1}$, which proves that $\mathfrak{X}_0$ is
non-compact.
\end{rmk}

In the rest of this paper, we assume that $G$ is finite. Then note that $G$ induces smooth actions on $$\mathcal{C}:=\mathcal{A}(W_+)\times \Gamma(W_+),$$ $$\mathcal{B}^*=(\mathcal{A}(W_+)\times (\Gamma(W_+)-\{0\}))/ \mathcal{G},$$ and also the Seiberg-Witten moduli space $\frak{M}$ whenever it is smooth.

Since $\frak{X}_\varepsilon$ is a subset of $\frak{M}_\varepsilon$, (actually a subset of the fixed locus $\mathfrak{M}^G_\varepsilon$ of a $G$-space $\frak{M}_\varepsilon$), we can define the
$G$-monopole invariant $SW^G_{M,\mathfrak{s}}$ by integrating the same universal cohomology classes as in the ordinary Seiberg-Witten invariant $SW_{M,\mathfrak{s}}$. Thus using the $\Bbb Z$-algebra isomorphism $$\mu_{M,\frak{s}} : \Bbb Z[H_0(M;\Bbb Z)]\otimes \wedge^*H_1(M;\Bbb Z)/\textrm{torsion}\ \tilde{\rightarrow}\ H^*(\mathcal{B}^*;\Bbb Z),$$ 
we define it as a function $$SW^G_{M,\frak{s}}: \Bbb Z[H_0(M;\Bbb
Z)]\otimes \wedge^*H_1(M;\Bbb Z)/\textrm{torsion}\rightarrow \Bbb
Z$$ $$\alpha\mapsto \langle [\frak{X}],\mu_{M,\frak{s}}(\alpha)
\rangle,$$ which is set to be 0 when the degree of
$\mu_{M,\frak{s}}(\alpha)$ does not match $\dim \frak{X}$. To be
specific, for  $[c]\in H_1(M,\Bbb Z)$,
$$\mu_{M,\frak{s}}([c]):=Hol_c^*([d\theta])$$ where $[d\theta]\equiv 1\in H^1(S^1,\Bbb Z)$ and  $Hol_c: \mathcal{B}^*\rightarrow S^1$ is given by the holonomy of each connection around $c$, and  $\mu_{M,\frak{s}}(U)$  for $U\equiv 1\in H_0(M,\Bbb Z)$ is given by the first
Chern class of the $S^1$-bundle
$$\mathcal{B}^*_o=(\mathcal{A}(W_+)\times (\Gamma(W_+)-\{0\}))/
\mathcal{G}_o$$ over $\mathcal{B}^*$ where $\mathcal{G}_o=\{g\in
\mathcal{G}| g(o)=1\}$ is the based gauge group for a fixed base
point $o\in M$. (The $S^1$-bundles obtained by choosing a
different base point are all isomorphic by the connectedness of
$M$.)

As in the ordinary case, a different choice of a $G$-invariant
metric and a $G$-invariant perturbation $\varepsilon$ gives a
cobordant $\frak{X}$ so that $SW^G_{M,\mathfrak{s}}$ is
independent of such choices, if $b_2^{+}(M)^G> 1$. When
$b_2^{+}(M)^G= 1$, one should get an appropriate wall-crossing
formula.


When $\frak{M}$ happens to be smooth for a $G$-invariant
perturbation, the induced $G$-action on it is a smooth action, and
hence $\mathfrak{M}^G$ is a smooth submanifold. Moreover if the
finite group action is free, then $\pi: M\rightarrow M/G$ is a
covering, and $\frak{s}$ is the pull-back of a Spin$^c$ structure
on $M/G$, which is determined up to
the kernel of $\pi^*: H^2(M/G,\Bbb Z)\rightarrow H^2(M,\Bbb Z),$ and  all the irreducible solutions of the upstairs is precisely the pull-back of the corresponding irreducible solutions of the downstairs :
\begin{thm}[\cite{RW, naka}]\label{nakamur}
Let $M$, $\mathfrak{s}$, and $G$ be as above. Under the assumption that $G$ is finite and
the action is free, for a $G$-invariant generic perturbation $$\frak{X}_{M,\mathfrak{s}}=\mathfrak{M}_{M/G,\mathfrak{s}'} \ \ \ \ \textrm{and} \ \ \ \ \mathfrak{M}^G_{M,\mathfrak{s}}\backsimeq\coprod_{c\in \ker \pi^*}\mathfrak{M}_{M/G,\mathfrak{s}'+c},$$ where the second one is a homeomorphism in general, and  $\mathfrak{s}'$ is the Spin$^c$ structure on $M/G$ induced from  $\frak{s}$ and its $G$-action.
\end{thm}

Finally we remark that the $G$-monopole invariant may change when
a homotopically different lift of the $G$-action to the Spin$^c$
structure is chosen.

\section{Connected sums and $G$-monopole invariant}

For $(\bar{M}_k,\bar{\frak{s}})$ described in Theorem
\ref{firstth}, there is at least one $G$-action lifted to
$\bar{\frak{s}}$ coming from the given $G$-action on
$(N,\frak{s}_N)$ and the $G$-equivariant gluing of $k$-copies
of $(M,\frak{s})$.
In general, there may be homotopically inequivalent liftings of the $G$-action on $\bar{M}_k$ to $\bar{\frak{s}}$.

Take a $G$-invariant metric of positive scalar curvature on $N$. In order to do the connected sum with $k$ copies of $M$, we perform a Gromov-Lawson type surgery \cite{GL,sung1} around each  point of a free orbit of $G$ keeping the positivity of scalar curvature to get a Riemannian manifold $\hat{N}$ with cylindrical ends with each end isometric to a Riemannian product of a round $S^3$ and $\Bbb R$. We suppose that this is done in a symmetric way so that the $G$-action on $\hat{N}$ is isometric.

On $M$ part, we put any metric and perform a Gromov-Lawson surgery with the same cylindrical end as above. Let's denote this by $\hat{M}$. Now chop the cylinders at sufficiently large length and then glue $\hat{N}$ and $k$-copies of $\hat{M}$ along the boundary to get a desired  $G$-invariant metric $g_k$ on $\bar{M}_k$. Sometimes we mean $(\bar{M}_k,g_k)$ by $\bar{M}_k$.

\begin{thm}
Let $(\bar{M}_k,\bar{\frak{s}})$ in Theorem
\ref{firstth} be endowed with $g_k$ as above. Then for any sufficiently large cylindrical length and some generic perturbation, $\mathfrak{X}_{\bar{M}_k,\bar{\mathfrak{s}}}$  is diffeomorphic to $\frak{M}_{M,\mathfrak{s}}\times T^{\nu}$, where  $\nu=\dim H_1(N;\Bbb R)^{G}$.
\end{thm}
\begin{proof}
First we consider the case when the $G$-action on $N$ has a fixed point.

Let's first figure out the ordinary moduli space $\frak{M}_{\bar{M}_k}$ of
$(\bar{M}_k,\bar{\mathfrak{s}})$. Let $\frak{M}_{\hat{M}}$ and
$\frak{M}_{\hat{N}}$ be the moduli spaces of finite-energy
solutions of Seiberg-Witten equations on $(\hat{M},\mathfrak{s})$ and
$(\hat{N},\mathfrak{s}_N)$ respectively. From now on, $[\ \cdot\ ]$ of a configuration $\cdot$ denotes its gauge equivalence class.

By the gluing theory\footnote{For more details, one may consult
\cite{KM, nicol, safari, vid1, sung2}.} of Seiberg-Witten moduli
space, which is now a standard method in gauge theory, $\frak{M}_{M}$
is diffeomorphic to $\frak{M}_{\hat{M}}$. In
$\frak{M}_{\hat{M}}$, we use a compact-supported self-dual 2-form
for a generic perturbation.

Since $\hat{N}$ has a metric of positive scalar curvature and the property that $b_2^+(\hat{N})=0$ and $c_1^2(\frak{s}_{\hat{N}})=-b_2(\hat{N})$,  $\hat{N}$ also has no gluing obstruction even without perturbation so that $\frak{M}_N$ is diffeomorphic to
$\frak{M}_{\hat{N}}=\frak{M}_{\hat{N}}^{red},$  which can be identified with the space of $L^2$-harmonic 1-forms on $\hat{N}$ modulo gauge, i.e. $$H^1_{cpt}(\hat{N},\Bbb R)/H^1_{cpt}(\hat{N},\Bbb Z)\simeq T^{b_1(N)}.$$ (Here by $T^0$ we mean a point, and $\frak{M}^{red}\subset \frak{M}$ denotes the moduli space of reducible solutions.)

As is well-known, approximate solutions on $\bar{M}_k$ are
obtained by chopping-off solutions on each $\hat{M}$ and $\hat{N}$
at a sufficiently large cylindrical length and then grafting them
to $\bar{M}_k$ via a sufficiently slowly-varying partition of
unity in a $G$-invariant way. More precise prescription of grafting is as follows. First, let's name $k$ $M$ parts of $\bar{M}_k$. Choose one of $k$ $M$ parts and we call it the 1st $M$ part. To assign other $k-1$ $M$ parts, let's denote $G$ by $\{\sigma_1, \sigma_2,\cdots,\sigma_k=e\}$ where $e$ is the identity element. Since each $M$ part of $\bar{M}_k$ is the image of the 1st $M$ part under exactly one of $\sigma_i\in G$, lets call it the $i$-th $M$ part. Now choosing an identification of the Spin$^c$ structure on the 1st $M$ part with that of $\hat{M}$, and the identifications of Spin$^c$ structures on other $M$ parts with that of $\hat{M}$ can be done using the $G$-action on $\bar{\frak{s}}$. Once there is such identification, we can graft a cut-off solution on $\hat{M}$ to each $M$ part of $\bar{M}_k$.

In taking cut-offs of solutions on $\hat{N}$, we use a special gauge-fixing condition. Fix a $G$-invariant connection $\eta_0$ such that $[\eta_0]\in\frak{M}_{\hat{N}}$,
which exists by taking the $G$-average of any reducible
solution, and take compact-supported closed 1-forms
$\beta_1,\cdots,\beta_{b_1(N)}$ which generate $H^1_{cpt}(\hat{N};\Bbb Z)$ and vanish on the cylindrical gluing regions. Any element $[\eta]\in\frak{M}_{\hat{N}}$ can be
expressed as $$\eta=\eta_0+\sum_{i}c_i\beta_i$$ for $c_i\in \Bbb
R/ \Bbb Z$, and the gauge equivalence class of its cut-off
$$\tilde{\eta}:=\rho\eta=\rho\eta_0+\sum_{i}c_i\beta_i$$ using a
$G$-invariant cut-off function $\rho$ which is equal to 1
on the support of every $\beta_i$ is well-defined independently of
the mod $\Bbb Z$ ambiguity of each $c_i$.

Similarly, for the cut-off procedure to be well-defined independently of the choice of a gauge representative  on $\frak{M}_{\hat{M}}$, one needs to take a gauge-fixing so that homotopy classes of gauge transformations on $\hat{M}$ are parametrized  by $H^1_{cpt}(\hat{M},\Bbb Z)$, whose elements are gauge transformations constant on gluing regions.
Thus the gluing produces a smooth map from $$(\prod_{i=1}^k\frak{M}_{\hat{M}})\times \frak{M}_{\hat{N}}:=(\underbrace{\frak{M}_{\hat{M}}\times \cdots \times
\frak{M}_{\hat{M}}}_k)\times \frak{M}_{\hat{N}}$$
to a so-called approximate moduli space
$\tilde{\frak{M}}_{\bar{M}_k}$ in $\frak{B}^*_{\bar{M}_k}$. This gluing map is one to one, because of the unique continuation principle (\cite{KM}) of Seiberg-Witten equations.  From the
unobstructedness of gluing, $\tilde{\frak{M}}_{\bar{M}_k}\subset
\frak{B}^*_{\bar{M}_k}$ is a smoothly embedded submanifold diffeomorphic to
\begin{eqnarray*}
((\prod_{i=1}^k\frak{M}_{\hat{M}}^o)/S^1)\times\frak{M}_{\hat{N}} &=&
((\prod_{i=1}^k\frak{M}_{\hat{M}})\tilde{\times} T^{k-1})\times T^{b_1(N)},
\end{eqnarray*}
where $\frak{M}_{\hat{M}}^o$ is the based moduli space fibering
over $\frak{M}_{\hat{M}}$ with fiber $\mathcal G_o/ \mathcal
G=S^1$, and $\tilde{\times}$ means a $T^{k-1}$-bundle over
$\prod_{i=1}^k\frak{M}_{\hat{M}}$.

As the length of the cylinders in $\bar{M}_k$ increases,
approximate solutions get close to genuine solutions
exponentially. Once we choose smoothly-varying normal subspaces to
tangent spaces of $\tilde{\frak{M}}_{\bar{M}_k}\subset
\frak{B}^*_{\bar{M}_k}$, the Newton method gives a diffeomorphism
$$\Upsilon : \tilde{\frak{M}}_{\bar{M}_k}\rightarrow
\frak{M}_{\bar{M}_k}$$ given by a very small isotopy along the
normal directions. A bit more explanation will be given in Lemma
\ref{saveme}.

An important fact for us is that the same $k$ copies of a compactly supported self-dual 2-form  can be used for the perturbation on $M$ parts, while no perturbation is put on the $N$ part. Along with the $G$-invariance of the Riemannian metric $g_k$, the perturbed Seiberg-Witten equations for $(\bar{M}_k,g_k)$ are $G$-equivariant so that the induced smooth $G$-action on $\mathcal{B}^*_{\bar{M}_k}$ maps $\frak{M}_{\bar{M}_k}$ to itself.

Let's describe elements of $\tilde{\frak{M}}_{\bar{M}_k}$ for
$(\bar{M}_k,g_k)$ more explicitly.  For $[\xi]\in
\frak{M}_{\hat{M}}$, let $\tilde{\xi}$ be an approximate solution
for $\xi$ cut-off at a large cylindrical length, and
$\tilde{\xi}(\theta)$ be its gauge-transform under the gauge
transformation by $e^{i\theta}\in C^\infty(\hat{M},S^1)$. (From
now on, the tilde $\tilde{\ }$ of a solution will mean its
cut-off.)  Any element in $\tilde{\frak{M}}_{\bar{M}_k}$ can be
written as an ordered $(k+1)$-tuple
\begin{eqnarray}\label{paul}
[(\tilde{\xi}_1(\theta_1),\cdots,\tilde{\xi}_{k-1}(\theta_{k-1}),\tilde{\xi}_k(0),
\tilde{\eta})]
\end{eqnarray}
 for each $[\xi_i]\in \frak{M}_{\hat{M}}$ and
constants $\theta_i$'s, where the $i$-th term for $i=1,\cdots , k$
represents the approximate solution grafted on the
$i$-th $M$ part, and the last term is a cut-off of $\eta\in
\frak{M}_{\hat{N}}^{red}$.  In fact, there is a bijective
correspondence
 \begin{eqnarray}\label{general}
\tilde{\frak{M}}_{\bar{M}_k}
\end{eqnarray}
\begin{eqnarray*}
\wr|
\end{eqnarray*}
$$\{ [(\tilde{\xi_1}(\theta_1),\cdots , \tilde{\xi}_{k-1}(\theta_{k-1}),\tilde{\xi_k}(0), \tilde{\eta})]\ |\ [\eta]\in \frak{M}_{\hat{N}},   [\xi_i]\in \frak{M}_{\hat{M}}, \theta_i\in [0,2\pi)\ \forall i
\}.
$$

\begin{lem}
The $G$-action on $\mathcal{B}^*_{\bar{M}_k}$ maps $\tilde{\frak{M}}_{\bar{M}_k}$ to itself.
\end{lem}
\begin{proof}
The $G$-action on $(\bar{M}_k,\bar{\frak{s}})$ can be obviously extended to an action on the Spin$^c$ structure of $\hat{N}\cup \amalg_{i=1}^k\hat{M}$ and also its moduli space of finite-energy monopoles.
Let $\sigma\in G$. By the $G$-invariance of $\rho$, $$\sigma^*\tilde{\eta}=\sigma^*(\rho\eta)=\rho\sigma^*\eta=\widetilde{\sigma^*\eta}.$$

Since $\sigma^*\beta_i$ also gives an element of $H^1_{cpt}(\hat{N};\Bbb Z)$, let's let $\sigma^*\beta_i$ be cohomologous to $\sum_j d_{ij}\beta_j$ for each $i$. Thus
$$\widetilde{\sigma^*\eta}=\rho\sigma^*(\eta_0+\sum_{i}c_i\beta_i)=\rho\eta_0+\sum_{i}c_i\sigma^*\beta_i$$ is gauge-equivalent to $$\rho\eta_0+\sum_{i,j}c_id_{ij}\beta_j$$ which is the cut-off of $\eta_0+\sum_{i,j}c_id_{ij}\beta_j.$

The $G$-action on  the first $k$ components of $(\ref{paul})$ just permutes them.
Thus  a constant gauge transform of $\sigma^*(\tilde{\xi}_1(\theta_1),\cdots,\tilde{\xi}_{k-1}(\theta_{k-1}),\tilde{\xi}_k(0),
\tilde{\eta})$ is also of the type (\ref{paul}).
\end{proof}

Moreover we may assume that the map $\Upsilon$ is $G$-equivariant by the following lemma.
\begin{lem}\label{saveme}
$\Upsilon$ can be made $G$-equivariant, and the smooth submanifold $\frak{M}_{\bar{M}_k}^{G}$
pointwisely fixed under the action is isotopic to
$\tilde{\frak{M}}_{\bar{M}_k}^{G}$, the fixed point set in
$\tilde{\frak{M}}_{\bar{M}_k}$.
\end{lem}
\begin{proof}
To get a $G$-equivariant $\Upsilon$, we need to choose a
smooth normal bundle of
$\tilde{\frak{M}}_{\bar{M}_k}\subset\frak{B}^*_{\bar{M}_k}$ in a $G$-equivariant way. This can be achieved by taking the $G$-average of any smooth Riemannian metric defined in a small
neighborhood of $\tilde{\frak{M}}_{\bar{M}_k}$.

A smooth Riemannian metric on a Hilbert manifold is a smoothly
varying bounded positive-definite symmetric bilinear forms on its tangent spaces. In
order to have a well-defined exponential map as a diffoemorphism
on a neighborhood of the origin, we want the metric to be
``strong" in the sense that the metric on each tangent space
induces the same topology as the original Hilbert space topology.
(For a proof, see \cite{kling}).

Since $\tilde{\frak{M}}_{\bar{M}_k}$ is compact, we use a
partition of unity on it to glue together obvious Hilbert space
metrics in local charts, thereby constructing a smooth Riemannian
metric in a neighborhood of $\tilde{\frak{M}}_{\bar{M}_k}$ in a
Hilbert manifold $\frak{B}^*_{\bar{M}_k}$. Taking its average under the $G$-action, we get a desired Riemannian metric, which is easily
checked to be strong.

Taking the orthogonal complement to the tangent bundle of
$\tilde{\frak{M}}_{\bar{M}_k}$ under the above-obtained metric, we
get its normal bundle which is trivial by being
infinite-dimensional. In the same way as the finite dimensional case, the inverse function theorem implies that a small neighborhood of the zero section in the normal bundle is mapped
diffeomorphically into $\frak{B}^*_{\bar{M}_k}$ by the exponential map. Thus we can view a small
neighborhood of $\tilde{\frak{M}}_{\bar{M}_k}$ as
$\tilde{\frak{M}}_{\bar{M}_k}\times \Bbb H$  where  $\Bbb H$ is
the Hilbert space isomorphic to the orthogonal complement of the
tangent space of $\tilde{\frak{M}}_{\bar{M}_k}$ at any point.

Applying the Newton method, $\Upsilon$ is pointwisely a vertical
translation along $\Bbb H$
direction. Now the assertion follows from the $G$-invariance of the normal directions.
\end{proof}

As a preparation for finding $G$-fixed points of $\tilde{\frak{M}}_{\bar{M}_k}$,
\begin{lem}\label{adam}
$\frak{M}_{\hat{N}}^{G}$ is diffeomorphic to $T^{\nu}$, the space of $G$-invariant $L^2$-harmonic 1-forms on $\hat{N}$ modulo $\Bbb Z$.
\end{lem}
\begin{proof}
Let $[\eta]\in
\frak{M}_{\hat{N}}^{G}$, i.e. $[\sigma^*\eta]=[\eta]$ for any $\sigma\in G$. Then
$$\bar{\eta}:=\frac{1}{k}\sum_{\sigma\in G}(\sigma)^*\eta$$ satisfies that $\sigma^*\bar{\eta}=\bar{\eta}$ for any $\sigma\in G$,  and $\bar{\eta}$ is cohomologous to $\eta$ so that  $[\bar{\eta}]=[\eta]$.

When $\nu\ne 0$,  let $b_1,\cdots,b_{b_1(N)}\in H_1(N;\Bbb Z)$ be a basis of $H_1(N;\Bbb R)$ such that  $b_1,\cdots,b_\nu\in H_1(N;\Bbb Z)^G$, where we used that $$\textrm{rank}(H_1(N;\Bbb Z)^{G})=\dim H_1(N;\Bbb R)^{G},$$ simply because $G$ also acts on $H_1(N;\Bbb Z)$. Also let $b_1^*,\cdots,b_{b_1(N)}^*\in H^1_{cpt}(\hat{N};\Bbb R)$ be the corresponding dual cohomology classes under the isomorphism $$H^1_{cpt}(\hat{N};\Bbb R)\simeq H_1(N;\Bbb R)^*.$$ Since
$b_i^*(b_j)=\delta_{ij}$ for all $i,j=1,,\cdots,b_1(N)$, a simple Linear algebra shows that
$b_1^*,\cdots,b_\nu^*$ are not only in $H^1_{cpt}(\hat{N};\Bbb Z)^{G}$, but also
form a basis of $H^1_{cpt}(\hat{N};\Bbb R)^{G}$. Therefore $\frak{M}_{\hat{N}}^{G}$  is a
$\nu$-dimensional torus spanned by $b_1^*,\cdots,b_\nu^*$.
When $\nu=0$, $\frak{M}_{\hat{N}}^{G}$  is a point.
\end{proof}

As an easy case,
\begin{lem}
If $G=\Bbb Z_k$, then $\tilde{\frak{M}}_{\bar{M}_k}^{\Bbb Z_k}$ is diffeomorphic to $k$ copies of $\frak{M}_{{M}}\times T^{\nu}$, where $T^0$ means a point.
\end{lem}
\begin{proof}
Let $\sigma$ be a generator of $\Bbb Z_k$, and take the numbering of elements of $G=\{\sigma_1,\cdots,\sigma_k=e\}$ such that $\sigma_i=\sigma^i$.
Thus the condition for a fixed point is that
$$(\tilde{\xi}_k(0),\tilde{\xi}_1(\theta_1),\cdots,
\tilde{\xi}_{k-1}(\theta_{k-1}),\widetilde{\sigma^*\eta})\equiv
(\tilde{\xi}_1(\theta_1),\cdots,\tilde{\xi}_{k-1}(\theta_{k-1}),\tilde{\xi}_k(0),\tilde{\eta}
)$$ modulo gauge transformations. By (\ref{general}) this implies
$$[\xi_{1}]=[\xi_{2}]=\cdots =[\xi_{k}]   \in\frak{M}_{\hat{M}},\ \textrm{and }\ [\sigma^*{\eta}]=[{\eta}]\in \frak{M}_{\hat{N}},$$
and
$$0 \equiv \theta_1+\theta,\ \ \theta_1 \equiv \theta_2+\theta,\cdots, \theta_{k-1}\equiv 0+\theta\ \ \  \textrm{mod}\ 2\pi$$
for some constant $\theta\in [0,2\pi)$. Summing up the above $k$ equations gives $$0\equiv k\theta\ \ \  \textrm{mod}\ 2\pi,$$ and hence
$$\theta=0,\frac{2\pi}{k},\cdots,\frac{2(k-1)\pi}{k},$$ which lead to the corresponding  $k$ solutions
\begin{eqnarray}\label{temp}
[(\tilde{\xi}((k-1)\theta),\tilde{\xi}((k-2)\theta),\cdots
,\tilde{\xi}(\theta) ,
\tilde{\xi}(0),\tilde{\eta})],
\end{eqnarray}
where we let $\xi_i=\xi$ for all $i$ and $[\eta]\in \frak{M}_{\hat{N}}^{\Bbb Z_k}$.
Therefore $\tilde{\frak{M}}_{M_k}^{\Bbb Z_k}$ is diffeomorphic to $k$ copies of $\frak{M}_{\hat{M}}\times\frak{M}_{\hat{N}}^{\Bbb Z_k}\simeq\frak{M}_{M}\times T^{\nu}.$
\end{proof}

\begin{lem}
If $G=\Bbb Z_k$, then $\mathfrak{X}_{\bar{M}_k}$  is diffeomorphic to $\frak{M}_{M}\times T^{\nu}$.
\end{lem}
\begin{proof}
For  $\xi\in \frak{M}_{\hat{M}}$, $\eta\in \mathfrak{X}_{\hat{N}}$, and $\theta$ as above, let  $$\tilde{\Xi}_{\theta}=(\tilde{\xi}((k-1)\theta),\tilde{\xi}((k-2)\theta),\cdots,\tilde{\xi}(\theta), \tilde{\xi}(0),\tilde{\eta}),$$  and denote $\Upsilon([\tilde{\Xi}_\theta])$ by $[\Xi_\theta]$. From the above lemma, we have that $$\sigma^*\tilde{\Xi}_\theta=e^{i\theta}\cdot\tilde{\Xi}_\theta,$$ where $\sigma$ is a generator of $\Bbb Z_k$, and $\cdot$ denotes the gauge action.

We will show that $k-1$ copies of  $\frak{M}_{M}\times T^{\nu}$ corresponding to nonzero $\theta$ do not belong to  $\mathfrak{X}_{\bar{M}_k}$.
Let $\theta=\frac{2\pi}{k},\cdots,\frac{2(k-1)\pi}{k}$. By the $\Bbb Z_k$-equivariance of $\Upsilon$,
$[\sigma^*\Xi_\theta]=\sigma^*[\Xi_\theta]$, and so write  $$\sigma^*\Xi_\theta=e^{i\vartheta}\cdot\Xi_\theta \ \ \ \textrm{for}\ e^{i\vartheta}\in Map(\bar{M}_k,S^1).$$ By taking the cylindrical length sufficiently large,  $e^{i\vartheta}$ can be made arbitrarily close to the constant $e^{i\theta}$ in a Sobolov norm and hence $C^0$-norm too by the Sobolov embedding theorem. (The Sobelev embedding constant does not change, if the cylindrical length gets large, because the local geometries remain unchanged.)


Assume to the contrary that $\sigma^*(\frak{g}\cdot\Xi_\theta)=\frak{g}\cdot\Xi_\theta$ for some $\frak{g} \in \mathcal{G}$.
Then combined with that
\begin{eqnarray*}
\sigma^*(\frak{g}\cdot\Xi_\theta)&=& \sigma^*(\frak{g})\cdot\sigma^*(\Xi_\theta)\\ &=& \sigma^*(\frak{g})\cdot (e^{i\vartheta}\cdot\Xi_\theta)\\ &=& (\sigma^*(\frak{g})e^{i\vartheta})\cdot \Xi_\theta,
\end{eqnarray*}
it follows that
$$\frak{g}\cdot \Xi_\theta = (\sigma^*(\frak{g})e^{i\vartheta})\cdot \Xi_\theta,$$
which implies that
\begin{eqnarray}\label{prayer2}
\sigma^*(\frak{g})=\frak{g}e^{-i\vartheta},
\end{eqnarray}
where we used the continuity of $g$ and the fact that the spinor part of $\alpha$ is not identically zero on an open subset by the unique continuation property.

Choose a fixed point $p\in \bar{M}_k$ under the $\Bbb
Z_k$-action.\footnote{This and the next two paragraphs are the only three places where we use the condition
that the action on $N$ has a fixed point, which was assumed in the
beginning of the proof of current theorem.} Then evaluating
(\ref{prayer2}) at the point $p$ gives $$\frak{g}(p)=\frak{g}(p)e^{-i\vartheta(p)}\approx \frak{g}(p)e^{-i\theta},$$ which yields a desired contradiction.

It remains to show that $\frak{M}_{M}\times T^{\nu}$ corresponding to $\theta=0$ belongs to $(\mathcal{A}(W_+)^G\times (\Gamma(W_+)^G-\{0\}))/ \mathcal{G}^G$.  Let $\Xi_0=\tilde{\Xi}_0+(a,\varphi)$ where $a\in \Gamma(\Lambda^1(\bar{M}_k;i\Bbb R))$ satisfies the Lorentz gauge condition $d^*a=0$. Since
$$\sigma^*\Xi_0=\tilde{\Xi}_0+(\sigma^*a,\sigma^*\varphi)$$ belongs to the same gauge equivalence class as  $\Xi_0$, and $$d^*(\sigma^*a)=\sigma^*(d^*a)=0$$ using the isometric action of $G$, we have that $\sigma^*a\equiv a$ modulo $H^1(\bar{M}_k;\Bbb Z)=\Bbb Z^{b_1(\bar{M}_k)}$. Applying the obvious identity $(\sigma^*)^k=\textrm{Id}$, it follows that $\sigma^*a=a$. This implies that $\sigma^*\Xi_0$ is a constant gauge transform $e^c\cdot \Xi_0$ of $\Xi_0$. If $e^c\ne 1$, it leads to a contradiction by the same method as above using the existence of a fixed point. Therefore $\sigma^*\Xi_0=\Xi_0$ as desired, and we conclude that  $\mathfrak{X}_{\bar{M}_k}$ is equal to $\frak{M}_{M}\times T^{\nu}$ .
\end{proof}

Now we will consider the case of any finite group $G$. We will show that $\mathfrak{X}_{\bar{M}_k}$ is diffeomorphic to $\Upsilon(S)$ where
$$S:=\{[(\tilde{\xi}(0),\cdots , \tilde{\xi}(0),\tilde{\xi}(0), \tilde{\eta})]\ |\ [\eta]\in \frak{M}_{\hat{N}}^G,   [\xi]\in \frak{M}_{\hat{M}}, \theta_i\in [0,2\pi)\ \forall i
\}.
$$
Since $(\tilde{\xi}(0),\cdots , \tilde{\xi}(0),\tilde{\xi}(0), \tilde{\eta})$ is $G$-invariant, $\Upsilon([(\tilde{\xi}(0),\cdots , \tilde{\xi}(0),\tilde{\xi}(0), \tilde{\eta})])$ is also represented by a $G$-invariant element by the same method as the above paragraph using the existence of a fixed point. Hence $\Upsilon(S)\subset \mathfrak{X}_{\bar{M}_k}$.


To show the reverse inclusion, first note that any element of  $\mathfrak{X}_{\bar{M}_k}$ can be written as $\Upsilon([(\tilde{\xi}(\theta_1),\cdots , \tilde{\xi}(\theta_{k-1}),\tilde{\xi}(0), \tilde{\eta})])$ for $[\xi]\in \frak{M}_{\hat{M}}$. We only need to show all $\theta_i$ are zero, and $[\eta]\in \frak{M}_{\hat{N}}^G$. For  $\sigma_i\in G$, let $\langle\sigma_i\rangle$ be the cyclic subgroup generated by $\sigma_i$, and $\mathfrak{X}_{\bar{M}_k,\langle\sigma_i\rangle}$  be the $\langle\sigma_i\rangle$-monopole moduli space. Since $\mathfrak{X}_{\bar{M}_k}$ is a subset of $\mathfrak{X}_{\bar{M}_k,\langle\sigma_i\rangle}\subset \frak{M}_{\bar{M}_k}$, we can use the above lemma to deduce that $\theta_i$ is 0, and $[\eta]\in \frak{M}_{\hat{N}}^{\langle\sigma_i\rangle}$. Since $i$ is arbitrary, we get a desired conclusion.

Finally let's prove the theorem when the action on $N$ is free. In this case, directly from Theorem \ref{nakamur} and the gluing theory, we have diffeomorphisms
\begin{eqnarray*}
\frak{X}_{\bar{M}_k,\bar{\mathfrak{s}}}&=& \mathfrak{M}_{M\# N/G,\mathfrak{s}\# \frak{s}_N'}\\
 &\simeq& \mathfrak{M}_{M,\mathfrak{s}}\times \mathfrak{M}^{red}_{N/G, \frak{s}_N'}\\
 &\simeq& \mathfrak{M}_{M,\mathfrak{s}}\times T^\nu,
\end{eqnarray*}
where $\frak{s}_N'$ is the Spin$^c$ structure on $N/G$ induced from $\frak{s}_N$ and its $G$ action induced from that of $\bar{\frak{s}}$. This completes all the proof.
\end{proof}

Now we come to the main theorem which implies Theorem \ref{firstth}.
\begin{thm}\label{myLord}
Let $(\bar{M}_k,\bar{\frak{s}})$ be as in Theorem \ref{firstth} and $d\geq 0$ be an integer.

If $\nu:=\dim H_1(N;\Bbb R)^{G}=0$, then for $A=1$ or
$a_1\wedge\cdots\wedge a_{j}$
$$SW^{G}_{\bar{M}_k,\bar{\frak{s}}}(U^d A)\equiv
SW_{M,\frak{s}}(U^d A)\ \ \ mod \ 2,$$ where $U$ denotes the
positive generator of the zeroth homology of $\bar{M}_k$ or $M$,
and each $a_i\in H_1(M;\Bbb Z)/\textrm{torsion}$ also denotes any
of $k$ corresponding elements in $H_1(\bar{M}_k;\Bbb Z)$ by abuse
of notation.

If $\nu\ne 0$, then
$$SW^{G}_{\bar{M}_k,\bar{\frak{s}}}(U^d A\wedge b_1\wedge\cdots\wedge b_\nu)\equiv
SW_{M,\frak{s}}(U^d A)\ \ \ mod \ 2,$$ where $A$ is as above, and
$b_1,\cdots, b_\nu\in H_1(N;\Bbb Z)$ is a basis of $H_1(N;\Bbb R)^{G}$.
\end{thm}
\begin{proof}
As before, let's first consider the case when the action has a fixed point. We continue to use the same notation and context as the previous theorem.
\begin{lem}\label{LHW}
The $\mu$ cocycles on $\frak{M}_{{M}}\times T^\nu$ and  $\mathfrak{X}_{\bar{M}_k}$ coincide, i.e.
$$\mu_{M}(a_i)=\mu_{\bar{M}_k}(a_i),\ \ \ \  \mu_{N}(b_i)=\mu_{\bar{M}_k}(b_i),\ \ \ \ \mu_{M}(U)=\mu_{\bar{M}_k}(U)$$  where the equality means the identification under the above diffeomorphism.
\end{lem}
\begin{proof}
The first equality comes from that the holonomy
maps $Hol_{a_i}$ defined on $\frak{M}_{{M}}$ and
$\tilde{\frak{M}}_{\bar{M}_k}^{G}$ are just the same, when
the representative of $a_i$ is  chosen away from the gluing
regions.  Using the isotopy between  $\frak{M}_{\bar{M}_k}^{G}$ and $\tilde{\frak{M}}_{\bar{M}_k}^{G}$, the induced maps $Hol^*_{a_i}$ from $H^1(S^1;\Bbb Z)$ to $H^1(\frak{M}_{{M}};\Bbb Z)$ and
$H^1(\frak{M}_{\bar{M}_k}^{G};\Bbb Z)$ are the same so that
$$\mu_{M}(a_i)=Hol^*_{a_i}([d\theta])=\mu_{\bar{M}_k}(a_i)$$
for each $i$. Likewise for the second equality.

For the third equality, note that the $S^1$-fibrations on
$\frak{M}_{\hat{M}}\times T^\nu$ and
$\tilde{\frak{M}}_{\bar{M}_k}^{G}$ induced by  the
$\mathcal{G}/\mathcal{G}_o$ action are isomorphic in an obvious
way, where the $T^\nu$ part is fixed under the
$\mathcal{G}/\mathcal{G}_o$ action. Since the isotopy between
$\tilde{\frak{M}}_{\bar{M}_k}$ and $\frak{M}_{\bar{M}_k}$ can be
extended to the $S^1$-fibrations  induced by  the
$\mathcal{G}/\mathcal{G}_o$ action, those $S^1$-fibrations are
isomorphic. In the same way using gluing theory, there are
isomorphisms of $S^1$-fibraions on  $\frak{M}_{M}$, its approximate moduli space
$\tilde{\frak{M}}_{M}$, and $\frak{M}_{\hat{M}}$. Therefore we
have an isomorphism between those $S^1$-fibrations on
$\frak{M}_{M}\times T^\nu$ and $\mathfrak{X}_{\bar{M}_k}$.
\end{proof}

We are ready for the evaluation of the Seiberg-Witten invariant
on $\mathfrak{X}_{\bar{M}_k}$.  Suppose $\nu\ne 0$. Let
$l_1,\cdots,l_{b_1(N)}$ be loops representing homology classes $b_1,\cdots,b_{b_1(N)}$
respectively.  Then $b_i^*$ introduced in Lemma \ref{adam}
restricts to a nonzero element  of
$H^1(l_j;\Bbb Z)$ iff $i=j$. Moreover $b_i^*$ is a generator of $H^1(l_j;\Bbb Z)$, and hence
$\{\mu(b_1),\cdots,\mu(b_\nu)\}$ is a standard generator of the 1st
cohomology of $T^\nu\simeq \Bbb R\langle
b_1^*,\cdots,b_\nu^*\rangle/\Bbb Z\langle
b_1^*,\cdots,b_\nu^*\rangle$. Combining the fact that $\mu(b_1)\wedge
\cdots \wedge \mu(b_{\nu})$ is a generator of $H^\nu(T^{\nu};\Bbb
Z)$ with the above identification of $\mu$-cocycles, we can conclude that
$$SW^{G}_{\bar{M}_k,\bar{\frak{s}}}(U^dA\wedge
b_1\wedge\cdots\wedge b_\nu)\equiv SW_{M,\frak{s}}(U^dA)\ \ \
\textrm{mod}\ 2$$ for $A=1$ or $a_1\wedge\cdots\wedge a_j$. The
case of $\nu=0$ is just a special case.

When the action is free, the theorem is obvious from the identification $\frak{X}_{\bar{M}_k,\bar{\mathfrak{s}}}=\mathfrak{M}_{M\# N/G,\mathfrak{s}\# \frak{s}_N'}$.
\end{proof}

\begin{rmk}
If the diffeomorphism between $\frak{X}_{\bar{M}_k}$ and $\mathfrak{M}_{M}\times T^\nu$
is orientation-preserving, then $G$-monopole invariants and Seiberg-Witten invariants are exactly the same.
We conjecture that the diffeomorphism between $\frak{X}_{\bar{M}_k}$  and $\frak{M}_{M}\times T^{\nu}$ is orientation-preserving, when the homology orientations are appropriately chosen.

One may try to prove $\frak{X}_{\bar{M}_k}\simeq \frak{M}_{M}\times T^{\nu}$ by gluing $G$-monopole moduli spaces directly. But the above method of proof by gluing ordinary moduli spaces also shows that for $G=\Bbb Z_k$, $\frak{M}_{\bar{M}_k}^{\Bbb Z_k}$ is diffeomorphic to $k$ copies of $\frak{M}_{M}\times T^{\nu}$. Lemma \ref{LHW} is also true for any other component of $\frak{M}_{\bar{M}_k}^{\Bbb Z_k}$.
\end{rmk}


\section{Examples of $(N,\frak{s}_N)$ of Theorem \ref{firstth}}

In this section,  $G, H$ and $K$ denote compact Lie groups. Let's  recall some elementary facts on equivariant principal bundles.
\begin{defn}
A principal $G$ bundle $\pi : P \rightarrow M$ is said to be $K$-equivariant if $K$ acts left on
both $P$ and $M$ in such a way that

(1) $\pi$ is $K$-equivariant :
$$\pi(k\cdot p) = k\cdot\pi(p)$$ for all $k\in K$ and $p\in P$,

(2) the left action of $K$ commutes with the right action of $G$ :
$$k\cdot(p\cdot g) = (k\cdot p)\cdot g$$ for all $k\in K, p\in P$, and $g\in G$.
\end{defn}
If $H$ is a normal subgroup of $G$, then one can define a principal $G/H$ bundle $P/H$ by taking the fiberwise quotient of $P$ by $H$. Moreover if $P$ is $K$-equivariant under a left $K$ action, then there exists the induced $K$ action on $P/H$ so that $P/H$ is $K$-equivariant.

\begin{lem}\label{jacob}
Let $P$ and  $\tilde{P}$ be a principal $G$ and $\tilde{G}$ bundle
respectively over a smooth manifold $M$ such that $\tilde{P}$
double-covers $P$ fiberwisely. For a normal subgroup $H$
containing $\Bbb Z_2$ in both $\tilde{G}$ and $S^1$ where the
quotient of $\tilde{G}$ by that $\Bbb Z_2$ gives $G$, let
$$\tilde{P}\otimes_{H}S^1:=(\tilde{P}\times_M (M\times S^1))/ H$$
be the quotient of the fiber product of $\tilde{P}$ and the
trivial $S^1$ bundle $M\times S^1$ by $H$, where the right $H$
action is given by $$(p,(x,e^{i\vartheta}))\cdot h=(p\cdot h,
(x,e^{i\vartheta}h^{-1})).$$

Suppose that $M$ and $P$ admit a smooth $S^1$ action such that $P$
is $S^1$-equivariant.  Then a principal $\tilde{G}\otimes_{H}S^1$
bundle $\tilde{P}\otimes_{H}S^1$ is also $S^1$-equivariant by
lifting the action on $P$. In particular, any smooth $S^1$-action
on a smooth spin manifold lifts to its trivial Spin$^c$ bundle so
that the Spin$^c$ structure is $S^1$-equivariant.
\end{lem}
\begin{proof}
Any left $S^1$ action on $P$ can be lifted to $\tilde{P}$ uniquely
at least locally commuting with the right $\tilde{G}$ action. If
the monodromy is trivial for any orbit, then the $S^1$ action can
be globally well-defined on $\tilde{P}$, and hence on
$\tilde{P}\otimes_{H}S^1$, where the $S^1$ action on the latter
$S^1$ fiber can be any left action, e.g. the trivial action,
commuting with the right $S^1$ action.

If the monodromy is not trivial, it has to be $\Bbb Z_2$ for any
orbit,  because the orbit space is connected. In that case, we
need the trivial $S^1$ bundle $M\times S^1$ with an ``ill-defined"
$S^1$ action with monodromy $\Bbb Z_2$ defined as follows.

First consider the double covering map from $M\times S^1$ to
itself defined by  $(x,z)\mapsto (x,z^2)$. Equip the downstairs
$M\times S^1$ with the left $S^1$ action which acts on the base as
given and on the fiber $S^1$ by the  multiplication as complex
numbers. Then this downstairs action can be locally lifted to the
upstairs commuting with the right $S^1$ action. Most importantly,
it has $\Bbb Z_2$ monodromy as desired. Explicitly,
$e^{i\vartheta}$ for $\vartheta\in[0,2\pi)$ acts on the fiber
$S^1$ by the  multiplication of $e^{i\frac{\vartheta}{2}}$.
Combining this with the local action on $\tilde{P}$, we get a
well-defined $S^1$ action on $\tilde{P}\otimes_{H}S^1$, because
two $\Bbb Z_2$ monodromies are cancelled each other.

Once the $S^1$ action on $\tilde{P}\otimes_{H}S^1$ is globally well-defined, it commutes with the right $\tilde{G}\otimes_{H}S^1$  action, because the local $S^1$ action on $\tilde{P}\times S^1$ commuted with the right $\tilde{G}\times S^1$ action.

If $S^1$ acts on a smooth manifold, the orthonormal frame bundle is always $S^1$-equivariant under the action. Then by the above result  any $S^1$ action on a smooth spin manifold lifts to  the trivial Spin$^c$ bundle which is $(\textrm{spin bundle})\otimes_{\Bbb Z_2}S^1$.
\end{proof}

\begin{lem}\label{joseph}
Let $P$ be a flat principal $G$ bundle over a smooth manifold $M$ with a smooth $S^1$ action. Suppose that the action can be lifted to the universal cover $\tilde{M}$ of $M$. Then it can be also   lifted to $P$ so that $P$ is $S^1$-equivariant.
\end{lem}
\begin{proof}
For the covering map $\pi: \tilde{M}\rightarrow M$, the pull-back
bundle $\pi^*P$ is the trivial bundle $\tilde{M}\times G$.  By
letting $S^1$ act on the fiber $G$ trivially, $\pi^*P$ can be made
$S^1$-equivariant. For the deck transformation group $\pi_1(M)$,
$P$ is gotten by an element of  $\textrm{Hom}(\pi_1(M),G)$. Any
deck transformation acts on each fiber $G$  as the left
multiplication of a constant in $G$ so that it commutes with not
only the right $G$ action but also the left $S^1$ action which is
trivial on the fiber $G$. Therefore the $S^1$ action on $\pi^*P$
projects down to an $S^1$ action on $P$. To see whether this $S^1$
action commutes with the right $G$ action, it's enough to check
for the local $S^1$ action, which can be seen upstairs on
$\pi^*P$.
\end{proof}

\begin{lem}
On a smooth closed oriented 4-manifold $N$ with $b_2^+(N)=0$, any  Spin$^c$ structure $\frak{s}$ satisfies $$c_1^2(\frak{s})\leq -b_2(N),$$ and the choice of a Spin$^c$ structure $\frak{s}_N$ satisfying $c_1^2(\frak{s}_N)=-b_2(N)$ is always possible.
\end{lem}
\begin{proof}
If $b_2(N)=0$, it is obvious. The case of $b_2(N)>0$ can be seen
as follows.  Using Donaldson's theorem \cite{donal1,donal2}, we
diagonalize the intersection form $Q_N$ on $H^2(N;\Bbb
Z)/\textrm{torsion}$ over $\Bbb Z$ with a basis
$\{\alpha_1,\cdots,\alpha_{b_2(N)}\}$ satisfying
$Q_N(\alpha_i,\alpha_i)=-1$ for all $i$. Then for any Spin$^c$
structure $\frak{s}$, the rational part of $c_1(\frak{s})$ should
be of the form $$\sum_{i=1}^{b_2(N)}a_i\alpha_i$$ where each
$a_i\equiv 1$ mod 2, because $$Q_N(c_1(\frak{s}),\alpha)\equiv
Q_N(\alpha,\alpha)\ \ \ \ \textrm{mod}\ 2$$ for any $\alpha\in
H^2(N;\Bbb Z)$. Consequently $|a_i|\geq 1$ for all $i$ which means
$$c_1^2(\frak{s})=\sum_{i=1}^{b_2(N)}-a_i^2\leq -b_2(N),$$ and we
can get a Spin$^c$ structure $\frak{s}_N$ with
$$c_1(\frak{s}_N)\equiv\sum_i \alpha_i\ \ \ \ \textrm{modulo
torsion}$$ by tensoring any $\frak{s}$ with a line bundle $L$
satisfying  $$2c_1(L)+c_1(\frak{s})\equiv\sum_i \alpha_i\ \ \ \
\textrm{modulo torsion},$$ completing the proof.
\end{proof}

\begin{thm}\label{Nexam}
Let $X$ be one of $$S^4,\ \  \overline{\Bbb CP}_2,\ \ S^1\times
(L_1\#\cdots\#L_n),\ \ \textrm{and}\ \ \widehat{S^1\times L}$$
where  each $L_i$ and $L$ are quotients of $S^3$ by free actions
of finite groups, and $\widehat{S^1\times L}$ is the manifold
obtained from the surgery on $S^1\times L$ along an $S^1\times
\{pt\}$.

Then for any integer $l\geq 0$ and any smooth
closed oriented 4-manifold $Z$ with $b_2^+(Z)=0$  admitting a
metric of positive scalar curvature, $$X\ \#\ kl Z$$  satisfies the properties of $N$ with $G=\Bbb Z_k$ in
Theorem \ref{firstth}, where the Spin$^c$ structure of $X \# kl Z$ is given by gluing any Spin$^c$ structure $\frak{s}_X$ on $X$ and any Spin$^c$ structure $\frak{s}_Z$ on $Z$ satisfying $c_1^2(\frak{s}_X)=-b_2(X)$ and $c_1^2(\frak{s}_Z)=-b_2(Z)$ respectively.
\end{thm}
\begin{proof}
First, we will define  $\Bbb Z_k$ actions  preserving a metric of positive scalar curvature.
In fact, our actions on $X$ will be induced from such $S^1$ actions.

For $X=S^4$, one can take a $\Bbb Z_k$-action coming from a nontrivial action of $S^1\subset SO(5)$ preserving a round metric. In this case, one can choose a free action or an action with fixed points also.

If $X=\overline{\Bbb CP}_2$, then one can use the following actions for some integers $m_1, m_2$ :
\begin{eqnarray}\label{exam}
j\cdot [z_0,z_1,z_2]=[z_0,e^{\frac{2jm_1}{k}\pi i}z_1,e^{\frac{2jm_2}{k}\pi i}z_2]
\end{eqnarray}
for $j\in \Bbb Z_k$, which preserve the Fubini-Study metric and has at least 3 fixed points $[1,0,0], [0,1,0], [0,0,1]$.

Before considering the next example, recall that every finite
group acting freely on $S^3$ is in fact conjugate to a subgroup of
$SO(4)$, and hence its quotient 3-manifold admits a metric of
constant positive curvature. This follows from the well-known
result of G. Perelman. (See \cite{morgan-tian1, morgan-tian2}.)

In $S^1\times (L_1\#\cdots\#L_n)$, the action is defined as a rotation along the $S^1$-factor, which is
obviously free and preserves a product metric. By endowing $L_1\#\cdots\#L_n$  with a metric of positive scalar curvature via the Gromov-Lawson surgery \cite{GL}, $S^1\times (L_1\#\cdots\#L_n)$ has a desired metric.

Finally the above-mentioned $S^1$ action on
$S^1\times L$ can be naturally extended to $\widehat{S^1\times
L}$, and moreover the Gromov-Lawson surgery \cite{GL} on
$S^1\times\{pt\}$ produces an $S^1$-invariant metric of positive
scalar curvature. Its fixed point set is $\{0\}\times S^2$ in the attached $D^2\times S^2$.

Now $X\# kl Z$ has an obvious $\Bbb Z_k$-action induced
from that of $X$ and a $\Bbb Z_k$-invariant metric which has positive
scalar curvature again by the Gromov-Lawson surgery.

It remains to prove that the above $\Bbb Z_k$-action on $X \# kl Z$ can be lifted to the Spin$^c$ structure obtained by gluing the above $\frak{s}_X$  and  $\frak{s}_Z$.
For this, we will only prove that any such $\frak{s}_X$  is $\Bbb Z_k$-equivariant. Then one can glue $k$ copies of $lZ$ in an obvious $\Bbb Z_k$-equivariant way. Recalling that the $\Bbb Z_k$ action on $X$ actually comes from an $S^1$ action, we will actually show the $S^1$-equivariance of  $\frak{s}_X$ on $X$.

On $S^4$,  the unique Spin$^c$ structure is trivial.  Any smooth $S^1$ action on $S^4$ which is spin can be lifted its trivial Spin$^c$ structure by Lemma \ref{jacob}.

Any smooth $S^1$ action on $\overline{\Bbb CP}_2$ is uniquely lifted to its orthonormal frame bundle $F$, and any  Spin$^c$ structure on $\overline{\Bbb CP}_2$
satisfying $c_1^2=-1$ is the double cover $P_1$ and $P_2$  of $F\oplus P$ and $F\oplus P^*$ respectively, where $P$ is the principal $S^1$ bundle over $\overline{\Bbb CP}_2$ with $c_1(P)=[H]$ and $P^*$ is its dual. Note that there is a base-preserving diffeomorphism between $P$ and $P^*$ whose total space is $S^5$. Obviously the action
(\ref{exam}) is extended to $S^5\subset \Bbb C^3$ commuting with the principal $S^1$ action of the Hopf fibration. By Lemma \ref{jacob} the $S^1$-action can be lifted to  $P_i\otimes_{S^1} S^1$ in an $S^1$-equivariant way, which is isomorphic to $P_i$ for $i=1,2$.

In case of $S^1\times (L_1\#\cdots\#L_n)$, any Spin$^c$ structure is
the pull-back from $L_1\#\cdots\#L_n$, and satisfies $c_1^2=0=-b_2$. Because the
tangent bundle is trivial, a free $S^1$-action is obviously
defined on its trivial spin bundle. Then the action can be
obviously extended to any Spin$^c$ structure, because it is
pulled-back from $L_1\#\cdots\#L_n$.

\begin{lem}
$\widehat{S^1\times L}$ is a rational homology 4-sphere, and  $$H^2(\widehat{S^1\times L};\Bbb Z)=H_1(L;\Bbb Z).$$ Its universal cover is $(|\pi_1(L)|-1)S^2\times S^2$ where $0(S^2\times S^2)$ means $S^4$.
\end{lem}
\begin{proof}
Since the Euler characteristic is easily computed to be 2 from the
surgery description,  and $b_1(\widehat{S^1\times L})=b_1(L)=0$,
it follows that $\widehat{S^1\times L}$ is a rational homology
4-sphere.

By the universal coefficient theorem,
\begin{eqnarray*}
H^2(\widehat{S^1\times L};\Bbb Z)
&=&\textrm{Hom}(H_2(\widehat{S^1\times L};\Bbb Z),\Bbb Z)\oplus \textrm{Ext}(H_1(\widehat{S^1\times L};\Bbb
Z),\Bbb Z)\\ &=& H_1(\widehat{S^1\times L};\Bbb Z)\\ &=& H_1(L;\Bbb Z).
\end{eqnarray*}

The universal cover is equal to the manifold obtained from $S^1\times S^3$ by performing surgery along  $S^1\times \{ |\pi_1(L)|\ \textrm{points in } S^3\}$, and hence it must be $(|\pi_1(L)|-1)S^2\times S^2$.
\end{proof}

By the above lemma, there are $|H_1(L;\Bbb Z)|$ Spin$^c$
structures on $\widehat{S^1\times L}$, all of which are torsion to satisfy
$c_1^2=0=-b_2(\widehat{S^1\times L})$. Since any $S^1$ bundle on  $\widehat{S^1\times L}$ is flat, and the $S^1$-action on  $\widehat{S^1\times L}$ can be obviously lifted to its universal cover, Lemma \ref{joseph} says that any $S^1$ bundle is $S^1$-equivariant under the $S^1$ action.

By the construction, $\widehat{S^1\times L}$ is spin, and hence the trivial Spin$^c$ bundle is $S^1$-equivariant by Lemma \ref{jacob}.
Any other Spin$^c$ structure is given by the tensor product over $S^1$ of the trivial Spin$^c$ bundle and an $S^1$ bundle, both of which are $S^1$-equivariant bundles. Therefore any Spin$^c$ bundle of $\widehat{S^1\times L}$ is  $S^1$-equivariant.

This completes all the proof.
\end{proof}


\end{document}